\documentclass{amsart}
\usepackage{amsmath,amssymb,amsfonts,enumerate,amsthm,graphicx}

\newcommand{\lk}{\mbox{lk}\,}

\newtheorem{thm}{Theorem}[section]
\newtheorem{cor}[thm]{Corollary}

\newtheorem{prop}[thm]{Proposition}

\newtheorem{rem}[thm]{Remark}

\numberwithin{equation}{section}

\begin{document}
 \bibliographystyle{amsplain}
 \title{Cohen-Macaulay lexsegment complexes in arbitrary codimension}

\author{Hassan Haghighi}
\address{Hassan Haghighi\\Department of Mathematics, K. N. Toosi
     University of Technology, Tehran, Iran.}

\author{Siamak Yassemi}
\address{Siamak Yassemi\\School of Mathematics, Statistics \&
Computer Science, University of Tehran, Tehran Iran.}

\author{Rahim Zaare-Nahandi}
      \address{Rahim Zaare-Nahandi\\School of Mathematics, Statistics \&
      Computer Science, University of Tehran, Tehran, Iran.}
      \thanks{H. Haghighi was supported in part by a grant from K. N. Toosi University of Technology}
      \thanks{S. Yassemi and R. Zaare-Nahandi were supported in part by a grant from the University of Tehran}

 \thanks{Emails: haghighi@kntu.ac.ir, yassemi@ipm.ir, rahimzn@ut.ac.ir}



\begin{abstract}

\noindent We characterize pure lexsegment complexes which are
Cohen-Macaulay in arbitrary codimension.  More precisely, we prove
that any lexsegment complex is Cohen-Macaulay if and only if it is
pure and its one dimensional links are connected, and, a lexsegment
flag complex is Cohen-Macaulay if and only if it is pure and
connected. We show that any non-Cohen-Macaulay lexsegment complex is
a Buchsbaum complex if and only if it is a pure disconnected flag
complex. For $t\ge 2$, a lexsegment complex is strictly
Cohen-Macaulay in codimension $t$ if and only if it is the join of a
lexsegment pure disconnected flag complex with a $(t-2)$-dimensional
simplex. When the Stanley-Reisner ideal of a pure lexsegment complex
is not quadratic, the complex is Cohen-Macaulay if and only if it is
Cohen-Macaulay in some codimension. Our results are based on a
characterization of Cohen-Macaulay and Buchsbaum lexsegment
complexes by Bonanzinga, Sorrenti and Terai.
\end{abstract}
\maketitle

\section{Introduction}

Primary significance of lexsegment ideals comes from Macaulay's
result that for any monomial ideal there is a unique lexsegment
monomial ideal with the same Hilbert function (see \cite{M 27} and
\cite{Ba 82}). Recent studies on the topic began with the work of
Bigatti \cite{B 93} and Hulett \cite{H 93} on extremal properties of
lexsegment monomial ideals. Aramova, Herzog and Hibi showed that for
any squarefree monomial ideal there exists a squarefree lexsegment
monomial ideal with the same Hilbert function \cite{AHH 98}. In this
direction, some characterizations of pure, Cohen-Macaulay and
Buchsbaum complexes associated with squarefree lexsegment ideals was
given by Bonansinga, Sorrenti and Terai in \cite{BST 12}. As a
generalization of Cohen-Macaulay and Buchsbaum complexes, ${\rm
CM}_t$ complexes, were studied in \cite{HYZ 12}. These are pure
simplicial complexes which are Cohen-Macaulay in codimension $t$.
Naturally, one may ask for a characterization of ${\rm CM}_t$
lexsegment complexes. In this paper, using the behavior of ${\rm
CM}_t$ property under the operation of join of complexes \cite{HYZ
15}, we first provide some modifications of the results of
Bonanzinga, Sorrenti and Terai in \cite{BST 12}. Then, we
characterize ${\rm CM}_t$ lexsegment simplicial complexes. Our
characterizations are mostly in terms of purity and connectedness of
certain subcomplexes. In particular, it turns out a lexsegment
complex is Cohen-Macaulay if and only it is pure and its one
dimensional links are connected while for lexsegment flag complexes,
the Cohen-Macaulay property is equivalent to purity and
connectedness of the simplicial complex. The Buchsbaum property is
equivalent to being Cohen-Macaulay or a pure flag complex. A
non-Buchsbaum complex is ${\rm CM}_t$, $t\ge 2$, if and only if it
is the join of a Buchsbaum complex with a $(t-2)$-simplex. It also
appears that any ${\rm CM}_t$ lexsegment complex for which the
associated Stanley-Reisner ideal is generated in degree $d \ge 3$,
is indeed Cohen-Macaulay. Our proofs are heavily based on the
results in \cite{BS 09} and particularly, on results in \cite{BST
12}.

\section{Preliminaries and notations}

Let $R = k[x_1, \cdots, x_n]$ be the ring of polynomials in $n$
variables over a field $k$ with standard grading. Let $M_d$ be the
set of all squarefree monomials of degree $d$ in $R$. Consider the
lexicographic ordering of monomials in $R$ induced by the order $x_1
> x_2 > \cdots > x_n$. A squarefree lexsegment monomial ideal in degree $d$ is
an ideal generated by a lexsegment $L(u, v) = \{ w\in M_d: u \ge w
\ge v \}$ for some $u, v \in M_d$ with $u \ge v$.\\

Let $\Delta$ be a simplicial complex on $[n] = \{1, \cdots, n\}$
with the Stanley-Reisner ring $k[\Delta]$. Recall that for any face
$F \in \Delta$, the link of $F$ in $\Delta$ is defined as follows:
$$\lk_\Delta(F)=\{G\in\Delta|G\cup F \in\Delta,
G\cap F =\varnothing\}.$$

In the sequel by a complex we will always mean a simplicial complex.
When a complex has a quadratic Stanley-Reisner ideal, that is, it is
the independence complex of a graph,then it is called a flag
complex.\\

A complex is said to satisfy the $S_2$ condition of Serre if
$k[\Delta]$ satisfies the $S_2$ condition. Using \cite[Lemma
3.2.1]{S 81} and Hochster's formula on local cohomology modules, a
pure $(d-1)$-dimensional complex $\Delta$ satisfies the $S_2$
condition if and only if
$\widetilde{H}_0(\mathrm{link}_\Delta(F);k)=0$ for all $F\in \Delta$
with $\#F\le d-2$ (see \cite[page 4]{T 07}). Therefore, $\Delta$ is
$S_2$ if and only if it is pure and ${\rm link}_\Delta(F)$ is
connected whenever $F\in \Delta$ and ${\rm dim}({\rm
link}_\Delta(F))\ge 1$.\\

Let $t$ be an integer $0\le t\le {\rm dim}(\Delta)-1$. A pure
complex $\Delta$ is called ${\rm CM}_t$, or Cohen-Macaulay in
codimension $t$, over $k$ if the complex $\lk_\Delta(F)$ is
Cohen-Macaulay over $k$ for all $F\in\Delta$ with $\#F\ge t$. It is
clear that for any $j\ge i$, ${\rm CM}_i$ implies ${\rm CM}_j$. For
$t\ge 1$, a ${\rm CM}_t$ complex is said to be strictly ${\rm CM}_t$
if it is not ${\rm CM}_{t-1}$. A squarefree monomial ideal is called
${\rm CM}_t$ if the associated simplicial complex is ${\rm CM}_t$.
Note that from the results by Reisner \cite{R 76} and Schenzel
\cite{S 81} it follows that ${\rm CM}_{0}$ is the same as
Cohen-Macaulayness and ${\rm CM}_1$ is identical with the Buchsbaum
property. \\

A complex $\Delta$ is said to be lexsegment if the associated
Stanley-Reisner ideal $I_\Delta$ is a lexsegment ideal. Therefore,
$\Delta$ is a ${\rm CM}_t$ lexsegment complex if $I_\Delta$ is a
lexsegment ideal and $\Delta$ is ${\rm CM}_t$.

\section{${\rm CM}_t$ lexsegment complexes}

The following result plays a significant role in the study of ${\rm
CM}_t$ lexsegment complexes.

\begin{thm}\cite[Theorem
3.1]{HYZ 15}\label{join} Let $\Delta_1$ and $\Delta_2$ be two
complexes of dimensions $r_1-1$ and $r_2-1$, respectively. Then
\begin{itemize}

\item[(i)] The join complex $\Delta_1 \ast \Delta_2$ is Cohen-Macaulay
if and only if $\Delta_1$ and $\Delta_2$ are both Cohen-Macaulay.

\item[(ii)] If $\Delta_1$ is Cohen-Macaulay and $\Delta_2$ is
${\rm CM}_t$ for some $t\ge 1$, then $\Delta_1\ast \Delta_2$ is
${\rm CM}_{r_1+t}$ (independent of $r_2$). This is sharp, i.e., if
$\Delta_2$ is strictly ${\rm CM}_t$, then $\Delta_1\ast \Delta_2$ is
strictly ${\rm CM}_{r+t}$.
\end{itemize}

\end{thm}

Let $d\ge 2$ be an integer and let $u \ge v$ be in $M_d$,
$u=x_ix_{i_2}\cdots x_{i_d}$ with $i<i_2<\cdots <i_d$. Let $\Delta$
be a complex on $[n]$ such that $I_\Delta = (L(u, v)) \subset
k[x_1,\cdots, x_n]$. Let $\Delta_1$ be the simplex on
$[i-1]=\{1,\cdots , i-1\}$, and let $\Delta_2$ be the complex of the
lexsegment ideal generated by $L(u, v)$ as an ideal in
$k[x_i,\cdots, x_n]$. Then it is immediate that $\Delta = \Delta_1
\ast \Delta_2$. Observe that, $\Delta$ is disconnected if and only
if $\Delta_1 = \emptyset$ and $\Delta_2$ is disconnected. Similarly,
$\Delta$ is pure if and only if $\Delta_2$ is pure. Furthermore, by
Theorem \ref{join} we have the following corollary.

\begin{cor}\label{CR1}
With the notation and assumption as above the following statements
hold:

\begin{itemize}

\item[(i)] The complex $\Delta$ is Cohen-Macaulay
if and only if $\Delta_2$ is Cohen-Macaulay. In other words, to
check the Cohen-Macaulay property of $\Delta$ one may always assume
 $i = 1$.

\item[(ii)] Assume that $\Delta_2$ is strictly ${\rm CM}_{t}$ for some $t\ge 1$.
Then $\Delta$ is strictly ${\rm CM}_{t+i}$. In particular, a
characterization of ${\rm CM}_t$ squarefree lexsegment ideals with
$i=1$ uniquely provides a characterization of  ${\rm CM}_t$
lexsegment ideals.
\end{itemize}
\end{cor}

\begin{rem} Corollary \ref{CR1} substantially simplifies the
statements and proofs of \cite{BS 09} and \cite{BST 12}.\\
\end{rem}

Based on Corollary \ref{CR1}, unless explicitly specified, we will
assume that $i=1$.\\

Bonanzinga, Sorrenti and Terai \cite{BST 12} have given a
characterization of Cohen-Macaulay squarefree lexsegment ideals in
degree $d\ge 2$. We give an improved version of their
result. Our proof is extracted from their proof.\\

\begin{thm}
\cite[An improved version of Theorem 3.4]{BST 12}\label{CMd} Let $u
> v$ be in $M_d$, $I_\Delta = (L(u, v))$.  Then the following
statements are equivalent:
\begin{itemize}
\item[(i)] $\Delta$ is shellable;
\item[(ii)] $\Delta$ is Cohen-Macaulay;
\item[(iii)] $\Delta$ is $S_2$;
\item[(iv)] $\Delta$ is pure and $\lk_\Delta(F)$ is connected for all $F\in \Delta$
with ${\rm dim}(\lk_\Delta(F))=1$.
\end{itemize}
\end{thm}
\begin{proof} Clearly, (iii)$\Rightarrow$(iv). Thus, by \cite[Theorem 3.4]{BST 12}
of Bonanzinga, Sorrenti and Terai, we only need to prove (iv)
$\Rightarrow$ (i). As mentioned above, we may assume $i=1$. Checking
all cases from the proof of (iii) $\Rightarrow$ (iv) in
\cite[Theorem 3.4]{BST 12}, it reveals that when $\Delta$ is pure,
if $u$ and $v$ are not in the list (1),...,(7) in their theorem,
then one of the following cases (a) (with $d\ge 3$), (b), or (c)
specified in the their proof, may occur: $(a) \
\lk_\Delta([n]\setminus \{1,2,d+1,d+2\}) =
\langle\{1,2\},\{d+1,d+2\}\rangle,$\\
$(b) \ \lk_\Delta(\{n-d+1,\cdots, \hat{k}, \widehat{k+1},\cdots,
n\}) = \langle\{1,2\},\{1,3\},\cdots,\{1,n-d\},\{k,k+1\}\rangle,$\\
where $n-d+1 \le k \le n-1$,\\ (c) \ $\lk_\Delta(\{n-d+3,\cdots,
n\}) = \langle\{1,2\},\cdots,\{1,n-d-1\}, \{n-d,n-d+1\},
\{n-d,n-d+2\},
\{n-d+1,n-d+2\}\rangle).$\\
In all these cases, $\lk_\Delta(F)$ is disconnected for some $F\in
\Delta$ with ${\rm dim}(\lk_\Delta(F))=1$. Therefore, assuming (iv)
above, $u$ and $v$ will be in the list (1),...,(7) in their theorem.
Hence, by the proof of (iv) $\Rightarrow$ (i) in their theorem,
$\Delta$ is shellable.
\end{proof}

\begin{rem} It is known that for some flag complexes including the independence complex of
a bipartite graph or a chordal graph, the conditions $S_2$ and
Cohen-Macaulay-ness are equivalent \cite{HYZ 10}. Nevertheless, the
condition (d) is in general weaker than the $S_2$ property. For
example, if $\Delta = \langle \{1,2,3,4\}, \{1,5,6.7\}\rangle$, then
$\Delta$ satisfies the condition (d) but does not satisfy the $S_2$
property. Indeed, any one-dimensional face has a connected
one-dimensional link, but $\lk_\Delta(\{1\})$ is disconnected of
dimension 2.
\end{rem}

For $d=2$, the statement in Theorem \ref{CMd}(iv) could be
relaxed as follows. The assumption $i=1$ is still in order.\\

\begin{thm}\label{CM} Let $u
> v$ be in $M_2$, $I_\Delta = (L(u, v))$.  Then the following
statements are equivalent:
\begin{itemize}
\item[(i)] $\Delta$ is shellable;
\item[(ii)] $\Delta$ is Cohen-Macaulay;
\item[(iii)] $\Delta$ is $S_2$;
\item[(iv)] $\Delta$ is pure and connected.
\end{itemize}
\end{thm}

\begin{proof}
We only need to check (iv) $\Rightarrow$ (i). Once again, checking
the proof of \cite[Theorem 3.4]{BST 12}, the case (a) could not
occur for $d=2$. In case (b), it follows that $u = x_1x_{n-1}$, $v =
x_{n-2}x_n$ with $n>3$. Then, $\Delta = \langle \{1,2\},
\{n-1,n\}\rangle$ is disconnected. In case (c), it turns out that $u
= x_1x_{n-2}$, $v = x_{n-2}x_{n-1}$ with $n>4$. Then, $\Delta =
\langle \{1,2\}, \{n-2,n\}, \{n-1,n\}\rangle$ is again disconnected.
Therefore, assuming purity and connectedness of $\Delta$, $u$ and
$v$ will be in the list (1),...,(7) of \cite[Theorem 3.4]{BST 12}.
Hence, by the proof of (iv) $\Rightarrow$ (i) of the same theorem,
$\Delta$ is shellable.
\end{proof}

Let $u > v$ be in $M_d$, $I_\Delta = (L(u, v))$. Bonanzinga,
Sorrenti and Terai in \cite{BST 12} have shown that for $d\ge 3$,
$\Delta$ is Buchsbaum if and only if it is Cohen-Macaulay. The same
proof implies that this is the case for ${\rm CM}_t$ lexsegment
complexes.

\begin{prop}\label{CMdt} Let $u
> v$ be in $M_d$ with $d\ge 3$ and $I_\Delta = (L(u, v))$. Then for
any $t \ge 0$, $\Delta$ is ${\rm CM}_t$ if and only if $\Delta$ is
Cohen-Macaulay.
\end{prop}
\begin{proof}
Clearly any Cohen-Macaulay complex is ${\rm CM}_t$. Assume that
$\Delta$ is ${\rm CM}_t$. Then as noticed in the proof \cite[Thorem
4.1]{BST 12}, for $d\ge 3$, ${\rm depth}k[\Delta] \ge 2$. Hence
$\Delta$ is $S_2$. Therefore, by Theorem \ref{CM}, $\Delta$ is
Cohen-Macaulay.
\end{proof}

By Proposition \ref{CMdt} to check the ${\rm CM}_t$ property with
$t\ge 1$, for squarefree lexsegment ideals we should restrict to the case $d=2$.\\

We now drop the assumption $i=1$ and assume that $u > v$ are in
$M_2$, $u=x_ix_j$, $v=x_rx_s$ with $i<j$ and $r<s$. Let $I_\Delta =
(L(u, v)) \subset k[x_1,\cdots, x_n]$. Recall that if $\Delta_1$ is
the simplex on $[i-1]$, and $\Delta_2$ is the complex of the
lexsegment ideal generated by $L(u, v)$ as an ideal in
$k[x_i,\cdots, x_n]$. Then $\Delta = \Delta_1 \ast \Delta_2$.

\begin{thm}\label{CM1} Let $u > v$ be in $M_2$, $I_\Delta = (L(u, v))$.
Assume that $\Delta$ is not Cohen-Macaulay. Then the following
statements are equivalent:
\begin{itemize}
\item[(1)] $\Delta$ is Buchsbaum;
\item[(2)] One of the following conditions hold;
\begin{itemize}
\item [(a)] $ u = x_1 x_{n-2},  v = x_{n-2}x_{n-1}, n > 4;$
\item[(b)] $ u = x_1 x_{n-1},  v = x_{n-2} x_n , n > 3$.
\end{itemize}
\item[(3)] $\Delta_1 = \emptyset$ and $\Delta_2$ is pure.
\end{itemize}
Furthermore, in either of these equivalent cases, $\Delta$ is
disconnected.
\end{thm}
\begin{proof} The equivalence of (1) and (2) is given in \cite[Theorem 2.1]{BS
09}. Assuming (2), then $\Delta_1 = \emptyset$, and as shown in the
proof of Theorem \ref{CM}, $\Delta=\Delta_2$ is pure in both cases
(a) and (b). Thus (2) $\Rightarrow$ (3). Now if $\Delta_1=\emptyset$
and $\Delta_2$ is pure, as it was observed in the proof of Theorem
\ref{CM}, the only cases where $\Delta$ is pure but not
Cohen-Macaulay are the cases (a) and (b) above, which settles (3)
$\Rightarrow$ (2). For the last statement, observe that if $\Delta =
\Delta_2$ also happens to be connected, then by Theorem \ref{CM},
$\Delta$ is Cohen-Macaulay. But this is contrary to the assumption.
Hence $\Delta$ is disconnected.
\end{proof}

\begin{thm} Let $u > v$ be in $M_2$, $I_\Delta = (L(u, v))$. Let
$\Delta = \Delta_1 \ast \Delta_2$ be as above. Let $t \ge 2$. Assume
that $\Delta$ is not ${\rm CM}_{t-1}$. Then the following is
equivalent:

\begin{itemize}
\item[(i)] $\Delta$ is ${\rm CM}_t$;
\item[(ii)] One of the following conditions hold;
\begin{itemize}
\item [(a)] $ u = x_t x_{n-2},  v = x_{n-2}x_{n-1}, n > 4;$
\item[(b)] $ u = x_t x_{n-1},  v = x_{n-2} x_n , n > 3$.
\end{itemize}
\item[(iii)] $\Delta_1$ is of dimension $t-2$ and $\Delta_2$ is pure.
\item[(iv)] $\Delta_1$ is of dimension $t-2$ and $\Delta_2$ is Buchsbaum.
\end{itemize}

\end{thm}

\begin{proof} The equivalence of (ii), (iii) and (iv) follows by applying
Theorem \ref{CM1} to $\Delta_2$.
Assuming (iv), the statement (i) follows by Corollary \ref{CR1}. Now
assume (i). Let $u=x_ix_j$, $v=x_rx_s$. Then since $\Delta$ is ${\rm
CM}_t$, it is pure, and hence $\Delta_2$ is pure. But since $\Delta$
is not Cohen-Macaulay, $\Delta_2$ can not be Cohen-Macaulay. Thus by
Theorem \ref{CM1}, $\Delta_2$ is Buchsbaum but not Cohen-Macaulay.
Now since $t\ge 2$ and $\Delta$ is not ${\rm CM}_{t-1}$, it follows
that $\Delta \neq \emptyset$ and $i\ge 2$. Hence by \ref{CR1}
$\Delta = \Delta_1 \ast \Delta_2$ is ${\rm CM}_i$ but not ${\rm
CM}_{i-1}$. Therefore, $i=t$ and $\Delta_1$ is of dimension $t-2$.
\end{proof}

\section*{Acknowledgments}
Part of this note was prepared during a visit of Institut de
Math\'{e}matiques de Jussieu, Universit\'{e} Pierre et Marie Curie
by the second and the third author. This visit was supported by
Center for International Studies \& Collaborations(CISSC) and French
Embassy in Tehran in the framework of the Gundishapur project
27462PL on the Homological and Combinatorial Aspects of Commutative
Algebra. The third author has been supported by research grant no.
4/1/6103011 of University of Tehran.

\end{document}